\documentclass[10pt,a4paper]{article}
\usepackage[utf8]{inputenc}
\usepackage{amsfonts}
\usepackage{diagbox}
\usepackage{ulem}
\usepackage{amsmath,amssymb,amsthm, mathrsfs}
\usepackage{hyperref}
\usepackage{fancyhdr}
\usepackage[all]{xy}
\usepackage{ulem}
\usepackage[usenames,dvipsnames]{color}
\usepackage{makeidx}
\usepackage{graphicx}
\usepackage{amsthm}
\usepackage{multirow}
\usepackage{xcolor}
\usepackage{amsbsy}
\usepackage[colorinlistoftodos,textwidth=25mm]{todonotes}

\usepackage[margin=1in]{geometry}
\newtheorem{theorem}{Theorem}
\newtheorem{corollary}[theorem]{Corollary}
\newtheorem{lemma}[theorem]{Lemma}

\newtheorem{example}[theorem]{Example}

\newtheorem{remark}[theorem]{Remark}

\title{New results on linear permutation polynomials with coefficients in a subfield}

\author{
 E. J. García-Claro\footnote{corresponding author.}\\
  \small{Departamento de Matemáticas}\\
  \small{Universidad Autónoma Metropolitana-Iztapalapa, Ciudad de México, México}\\
  \small{C\'odigo postal 09340, Ciudad de M\'exico, M\'exico}\\
  \small{eliasjaviergarcia@gmail.com}
  \and
  Gustavo Terra Bastos\\
  \small{Department of Mathematics and Statistics}\\
  \small{Federal University of São João del-Rei, São João del-Rei, Minas Gerais}\\
  \small{Frei Orlando Square, 170,  Minas Gerais, São João del-Rei, 36307-334, Brazil}\\
  \small{gtbastos@ufsj.edu.br}
}

\begin{document}

\maketitle

\begin{abstract}

Some families of linear permutation polynomials of $\mathbb{F}_{q^{ms}}$ with coefficients in $\mathbb{F}_{q^{m}}$ are explicitly described (via conditions on their coefficients) as isomorphic images of classical subgroups of the general linear group of degree $m$ over the ring $\frac{\mathbb{F}_{q}[x]}{\left\langle x^{s}-1 \right\rangle}$.  In addition, the sizes of some of these families are computed. Finally, several criteria to construct linear permutation polynomials of $\mathbb{F}_{q^{2p}}$ (where $p$ is a prime number) with prescribed coefficients in $\mathbb{F}_{q^{2}}$ are given. Examples illustrating the main results are presented.\\
\end{abstract}

\textit{keywords}: linear polynomials, linearized polynomials, $q$-polynomials, permutation polynomials; linear permutation polynomials; general linear group; special linear group; Borel subgroup.\\


\section*{Introduction}
Let $\mathbb{F}_q$ denote the finite field with $q$ elements. A polynomial $f(x)\in \mathbb{F}_{q^{n}}[x]$ is called a permutation polynomial (PP) if the function $f:\mathbb{F}_{q^{n}} \rightarrow \mathbb{F}_{q^{n}} $ given by $a\mapsto f(a)$ for all $a\in \mathbb{F}_{q^{n}}$ is a bijection, i.e., $f$ is a permutation of $\mathbb{F}_{q^{n}}$. PPs play a crucial role in several research areas, including cryptography, as part of encryption algorithms (see \cite{ prince2012, menezes1996,  vaught2016}); and coding theory, in the study and construction of codes with desired properties (see \cite{ ding2013,gerike2020,mullen2013, zhao2019}).\\  

A linear (linearized) polynomial (or $q$-polynomial) over $\mathbb{F}_{q^{n}}$ is a polynomial  $f(x) = \sum_{i=0}^{n-1}f_i x^{q^{i}} \in \mathbb{F}_{q^{n}} [x]$. These polynomials are $\mathbb{F}_{q}$-linear endomorphisms of $\mathbb{F}_{q^{n}}$, and form an $\mathbb{F}_{q}$-algebra with the sum and composition of functions. Let $n,m,s\in \mathbb{Z}_{>0}$ and $n=ms$, then the polynomials with shape $f(x) = \sum_{i=0}^{n-1}f_i x^{q^{i}} \in \mathbb{F}_{q^{m}} [x]$ form a sub-algebra $R_{m}$ of the linear polynomials over $\mathbb{F}_{q^{n}}$.\\

Let $\mathcal{R}_{q,s}:=\frac{\mathbb{F}_{q} [x]}{\left\langle x^{s}-1 \right\rangle}$. In \cite{bcv, micheli2015, wuliu} it was proved, using distinct thecniques, that $R_{m}$ is isomorphic to the algebra $M_{m}(\mathcal{R}_{q,s})$ of the $m\times m$ matrices with coefficients in $\mathcal{R}_{q,s}$. This guarantees that PPs in $R_{m}$ are in one-one correspondence with elements of the general linear group $GL_{m}(\mathcal{R}_{q,s})$ of degree $m$ over $R_{q,s}$. When working with linear polynomials, criteria to determine whether they are permutations commonly consist of verifying if their coefficients satisfy certain conditions. Nonetheless, the isomorphisms  presented in \cite{bcv, micheli2015, wuliu} do not allow such descriptions. This kind of description has been made for the algebra of linear polynomials over $\mathbb{F}_{q^{n}}$ in \cite{ yuan2011,zhou2008}; but, up to our knowledge, it has not been made for the PPs in $R_{m}$ yet.\\

In this paper, our primary aims are to explicitly describe diverse families of PPs in $R_{m}$ that are images of subgroups of $GL_{m}(\mathcal{R}_{q,s})$ under the inverse of the isomorphism given in  \cite{bcv}, and to compute the sizes of some of these families of PPs.\\

The manuscript is organized as follows: In Section \ref{S1} a particular isomorphism of $\mathbb{F}_{q}$-algebras between $M_{m}(\mathcal{R}_{q,s})$ is constructed extending the approach given in \cite{bcv}. This isomorphism is later used in Section \ref{S2} to determine explicit forms of PPs in $R_{m}$ that are images of certain matrices in $M_{m}(\mathcal{R}_{q,s})$ (among which are elementary matrices). Using these results, in Section \ref{S3}, explicit descriptions of families of PPs that are  image of subgroups of $GL_{m}(\mathcal{R}_{q,s})$ are presented. The sizes of some of these families are later computed in Section \ref{S4}. Finally, in Section \ref{S5}, two particular constructions of linear PPs of $\mathbb{F}_{q^{2p}}$ (where $p$ is a prime number) with coefficients in $\mathbb{F}_{q^{2}}$ are given.

\section{An isomorphism between $M_{m}\left(\mathcal{R}_{q,s}\right)$ and $R_m$}\label{S1}

 In this section, an isomorphism between the  algebra  $M_{m}\left(\mathcal{R}_{q,s}\right)$ (of the $m\times m$ matrices with coefficients in $\mathcal{R}_{q,s}$)  and  $R_m$  is presented. This will be used later to compute distinct families of PPs in  $R_m$.

\begin{theorem}\cite[Theorem 3.1]{bcv}\label{thbcv}
  $R_m$ and the  matrix algebra  $M_{m}\left(\mathcal{R}_{q,s}\right)$ are isomorphic $\mathbb{F}_{q}$-algebras.  \\  
\end{theorem}

Since $R_m \cong M_{m}\left(\mathcal{R}_{q,s}\right)$ as $\mathbb{F}_{q}$-algebras, taking the image of invertible matrices in $M_{m}\left(\mathcal{R}_{q,s}\right)$ under a fixed isomorphism gives a way to construct permutation polynomials in $R_{m}$. Even so, the isomorphism provided in the proof of the Theorem \ref{thbcv} is not presented in a form that allows such kind of computations. In that Theorem,  an isomorphism $\varphi$ that goes from  $R_m$ to $M_{m}\left(\mathcal{R}_{q,s}\right) $ is introduced under the following context:

If $\displaystyle{g(x)=\sum_{i=0} ^{n-1} a_{i}x^{q^i}}\in R_{m}$, $g(x)$ may be rewritten as
\begin{equation}
g(x)=\sum_{i=0} ^{s-1}\sum_{r=0} ^{m-1} a_{im+ r}x^{q^{im+r}} = \sum_{i=0}^{s-1} g_i \left(x^{q^{mi}}\right), 
\end{equation}
where $\displaystyle{g_i (x) := \sum_{r = 0} ^{m-1} a_{im+r}x^{q^r}}$, for $i=1,2,...,s-1$. Thus, if $B$ is an ordered $\mathbb{F}_{q}$-basis of $\mathbb{F}_{q^m}$ and $\left[g_i (x) \right]_B$ denotes the matrix of $g_{i}(x)$ on the basis $B$ for $i=1,...,s$, then 
\begin{equation}
\begin{array}{cccc}
\varphi:& R_m  &\longrightarrow & M_{m}\left(\mathcal{R}_{q,s}\right) \\
& g(x)  &\mapsto & \varphi(g(x)) = \left[g_0 (x) \right]_B + \left[g_1 (x)\right]_B \cdot x + \cdots +\left[g_{s-1} (x)\right]_B \cdot x^{s-1}
\end{array}
\end{equation}
is an isomorphism of $\mathbb{F}_{q}$-algebras.\\

Now we are going to describe the isomorphism $\psi=\varphi^{-1}$. This will allow studying how the entries of an arbitrary matrix $A\in M_{m}\left(\mathcal{R}_{q,s}\right)$ influence the coefficients of the linear polynomial $\psi(A)\in R_m$, so that we may produce PPs in $R_{m}$ with desired forms just by taking adequate matrices in $M_{m}\left(\mathcal{R}_{q,s}\right)$. Using the same notation as before, if $B$ is an ordered $\mathbb{F}_{q}$-basis of $\mathbb{F}_{q^m}$, then 
\begin{equation}\label{mainiso}
\begin{array}{cccc}
\psi:& M_{m}\left(\mathcal{R}_{q,s}\right)   &\longrightarrow &  R_m\\
& G = G_{0} + G_{1} \cdot x + \cdots +G_{s-1} \cdot x^{s-1}  &\mapsto &  \displaystyle{\psi(G)=g(x):=\sum_{i=0}^{s-1} g_i \left(x^{q^{mi}}\right)}
\end{array}
\end{equation}
where $\displaystyle{g_i (x) = \sum_{r = 0} ^{m-1} a_{ir}x^{q^r}}$ and $[g_{i}]_{B}=G_{i}$ for $i=0,...,s-1$.\\

Thus, to know $\psi(G)$ for a given $G\in  M_{m}\left(\mathcal{R}_{q,s}\right) $, it is necessary to compute the coefficients of any polynomial with a form as the $g_{i}(x)$ for $i=0,...,s-1$. Since this kind of polynomial is $\mathbb{F}_{q}$-linear, computing their coefficients may be solved just by knowing the image of the elements in a $\mathbb{F}_{q}$-basis of $\mathbb{F}_{q^{m}}$. Theorem~\ref{gi-coefficients} uses dual normal bases (\cite[Definition 2.30]{finitefields}) to solve that computation in an easy way, so we are going to recall some definitions that will be necessary later. Let $\mathbf{F}$ be an extension of degree $m$ of a finite field $F$. If  $B=\{\alpha_{1} , \alpha_{2} ,\ldots , \alpha_{m} \}$ and $B'=\{u_{1}, u_{2},\ldots, u_{m}\}$ are $F$-bases of $\mathbf{F}$, then $B$ and $B'$ are said to be dual (or complementary) bases if for $1 \leq i,j \leq m$,

\[Tr_{\mathbf{F}/F} \left(\alpha_{i} u_{j} \right) = \begin{cases} 1   &\mbox{if } i=j  \\
0 & \mbox{if } i\neq j \end{cases}.\]

 An element $\alpha\in \mathbf{F}$ is called normal over $F$ (or $F$-normal) if the set $A=\{\alpha, \alpha^{q},..., \alpha^{q^{m-1}}\}$ is an  $F$-basis of $\mathbf{F}$ (\cite[Definition 2.32]{finitefields}); and in such case, $A$ is called a normal basis of $\mathbf{F}$ over $F$. A normal basis always exists, for any finite extension of a finite field (see \cite[Theorem 2.35]{finitefields}). In fact, this kind of basis can be built using, for example, \cite[Theorem 5.2.7]{gao}. In addition, the dual of a normal basis is also a normal basis (see  \cite[Theorem 2.2.7]{gao}).\\

\begin{theorem}\label{gi-coefficients}
Let $\alpha \in \mathbb{F}_{q^{m}}$ be an $\mathbb{F}_{q}$-normal element, $B=\left\{\alpha, \alpha^{q},..., \alpha^{q^{m-1}}\right\}$ be the ordered normal basis determined by $\alpha$, and $B'=\left\{u, u^{q},..., u^{q^{m-1}}\right \}$ the dual basis of $B$. Let $\displaystyle{h (x) = \sum_{r = 0} ^{m-1} a_{r}x^{q^r}}\in R_{m}$. Then

 \[
 \begin{bmatrix}
a_{0} \\ a_{1} \\ \vdots \\ a_{m-1}
\end{bmatrix}
= \begin{bmatrix}
u & u^{q} & \cdots & u^{q^{m-1}} \\
u^{q} & u^{q^{2}} & \cdots & u \\
\vdots  & \vdots  & \ddots & \vdots  \\
u^{q^{m-1}} &u & \cdots & u^{q^{m-2}} 
\end{bmatrix}
\begin{bmatrix}
h(\alpha) \\ h(\alpha^{q}) \\ \vdots \\ h(\alpha^{q^{m-1}}) 
\end{bmatrix}.
\]

\end{theorem}

\begin{proof}
Since $\displaystyle{h (\alpha^{q^{l}}) = \sum_{r = 0} ^{m-1} a_{r}\alpha^{q^{r+l}}}$, then

\[
\left[\alpha^{q^{l}},\alpha^{q^{1+l}},..., \alpha^{q^{m-1+l}}\right]
\begin{bmatrix}
a_{0} \\ a_{1} \\ \vdots \\ a_{m-1}
\end{bmatrix}= h(\alpha^{q^{l}})
\]

for $l=0,...,m-1$. Thus, if  

\[A=\begin{bmatrix}
\alpha & \alpha^{q} & \cdots & \alpha^{q^{m-1}} \\
\alpha^{q} & \alpha^{q^{2}} & \cdots & \alpha \\
\vdots  & \vdots  & \ddots & \vdots  \\
\alpha^{q^{m-1}} &\alpha & \cdots & \alpha^{q^{m-2}} 
\end{bmatrix},\]

then

\[
A\cdot 
\begin{bmatrix}
a_{0} \\ a_{1} \\ \vdots \\ a_{m-1}
\end{bmatrix}=
\begin{bmatrix} 
\alpha & \alpha^{q} & \cdots & \alpha^{q^{m-1}} \\
\alpha^{q} & \alpha^{q^{2}} & \cdots & \alpha \\
\vdots  & \vdots  & \ddots & \vdots  \\
\alpha^{q^{m-1}} &\alpha & \cdots & \alpha^{q^{m-2}} 
\end{bmatrix}
\begin{bmatrix}
a_{0} \\ a_{1} \\ \vdots \\ a_{m-1}
\end{bmatrix}= \begin{bmatrix}
h(\alpha) \\ h(\alpha^{q}) \\ \vdots \\ h(\alpha^{q^{m-1}}) 
\end{bmatrix}.
\]

On the other hand, since $B'=\left\{u, u^{q},..., u^{q^{m-1}}\right \}$ is the dual basis of $B=\{\alpha, \alpha^{q},..., \alpha^{q^{m-1}}\}$, then $A$ is invertible with inverse 

\[A^{-1}=\begin{bmatrix}
u & u^{q} & \cdots & u^{q^{m-1}} \\
u^{q} & u^{q^{2}} & \cdots & u \\
\vdots  & \vdots  & \ddots & \vdots  \\
u^{q^{m-1}} &u & \cdots & u^{q^{m-2}} 
\end{bmatrix}.\]
Therefore,  \[
 \begin{bmatrix}
a_{0} \\ a_{1} \\ \vdots \\ a_{m-1}
\end{bmatrix}
= A^{-1} \begin{bmatrix}
h(\alpha) \\ h(\alpha^{q}) \\ \vdots \\ h(\alpha^{q^{m-1}}) 
\end{bmatrix}=  \begin{bmatrix}
u & u^{q} & \cdots & u^{q^{m-1}} \\
u^{q} & u^{q^{2}} & \cdots & u \\
\vdots  & \vdots  & \ddots & \vdots  \\
u^{q^{m-1}} &u & \cdots & u^{q^{m-2}} 
\end{bmatrix}
\begin{bmatrix}
h(\alpha) \\ h(\alpha^{q}) \\ \vdots \\ h(\alpha^{q^{m-1}}) 
\end{bmatrix}.
\]

\end{proof}

Note that Theorem \ref{gi-coefficients} is useful to construct PPs via the isomorphism $\psi$. In particular, the computations are made easier  if $B$ is self-dual. Now we recall some results on self-dual normal bases that could be of use to apply distinct results in this work and to build examples. The finite field $\mathbb{F}_{q^{m}}$ admits a self-dual normal basis over $\mathbb{F}_{q}$ iff both $m$ and $q$ are odd or $2\mid q$ and $4\nmid m$ (\cite[ Theorem 1.4.4]{gao}). In \cite[Section 5.4]{gao} some methods to build self-dual normal bases are presented.

\begin{remark}
All linear polynomials studied in this work have their coefficients expressed in terms of normal dual bases of $\mathbb{F}_{q^{m}}$ over $\mathbb{F}_{q}$. Even so, if one wants to depict these coefficients in terms of a $\mathbb{F}_{q}$-basis of $\mathbb{F}_{q^{n}}$ it is enough to use the following fact: if $a\in \mathbb{F}_{q^{n}}$, $\{\alpha_{1},..., \alpha_{n}\}$ is an $\mathbb{F}_{q}$-basis of $\mathbb{F}_{q^{n}}$ with dual $\{\beta_{1},..., \beta_{n}\}$ , then $a=\sum_{i=1}^{n}Tr_{\mathbb{F}_{q^{n}}/\mathbb{F}_{q^{m}}}(a\cdot \beta_{i})\alpha_{i}$. 
\end{remark}

\begin{example}
 In this example, we present one linear PP $g(x)$ over $\mathbb{F}_{64}$ whose  coefficients are in $\mathbb{F}_8$, so $q=2$, $n=6$, $m=3$, and $s=2$. Let $\displaystyle{\alpha \in\mathbb{F}_{2^3}\cong\frac{\mathbb{F}_2[x]}{\langle x^3 + x^2 +1\rangle}}$ a primitive element. It is easy to check that $B=\left\{\alpha, \alpha^2 , \alpha^4 \right\}$ a $\mathbb{F}_2 -$self dual normal basis of $\mathbb{F}_{2^3}$.

Consider the $3\times 3-$binary matrices
\begin{eqnarray*}
G_0 =\begin{bmatrix}
    0 & 0 & 1 \\
    1 & 0 & 0 \\
    0 & 0 & 0
\end{bmatrix}\mbox{ and } G_1 =\begin{bmatrix}
    1 & 0 & 0 \\
    0 & 1 & 0 \\
    0 & 0 & 1
\end{bmatrix}  
\end{eqnarray*}

and  $G := G_0 + G_1 \cdot x \in M_3 \left(\mathcal{R}_{2,2}\right)$. Then, $det(G)=x$ so that $\psi (G) = g(x):= g_0 (x) + g_1 \left(x^8 \right)$, where $\displaystyle{g_i (x) = \sum_{k = 0} ^{2} a_{ik}x^{q^k}}$ and $G_{i}=\left[g_i \right]_{B}$ for $i=1,2$, is a PP. Then, by Theorem \ref{gi-coefficients}, the coefficients of $g_0 (x)$ and $g_1 (x)$ can be obtained as
\begin{eqnarray*}
\begin{bmatrix}
     a_{00}    \\
     a_{01} \\ 
     a_{02} 
\end{bmatrix}&=&\begin{bmatrix}
    \alpha   & \alpha^2 & \alpha^4\\
    \alpha^2 & \alpha^4 & \alpha\\
    \alpha^4 & \alpha & \alpha^2
\end{bmatrix}\cdot\begin{bmatrix}
     0\cdot \alpha +  1\cdot \alpha^2 +  0\cdot \alpha^4   \\
     0\cdot \alpha +  0\cdot \alpha^2 +  0\cdot \alpha^4 \\ 
     1\cdot \alpha +  0\cdot \alpha^2 +  0\cdot \alpha^4 
\end{bmatrix}\\
             &=&\begin{bmatrix}
    \alpha   & \alpha^2 & \alpha^4\\
    \alpha^2 & \alpha^4 & \alpha\\
    \alpha^4 & \alpha & \alpha^2
\end{bmatrix}\cdot \begin{bmatrix}
     \alpha^2    \\
     0 \\ 
      \alpha 
\end{bmatrix}=\begin{bmatrix}
     \alpha^6 \\
     \alpha^5 \\ 
      \alpha^5 
\end{bmatrix}\\
            &\text{and}&\\
\begin{bmatrix}
     a_{10}    \\
     a_{11} \\ 
     a_{12} 
\end{bmatrix}&=&\begin{bmatrix}
    \alpha   & \alpha^2 & \alpha^4\\
    \alpha^2 & \alpha^4 & \alpha\\
    \alpha^4 & \alpha & \alpha^2
\end{bmatrix}\cdot\begin{bmatrix}
     1\cdot \alpha +  0\cdot \alpha^2 +  0\cdot \alpha^4   \\
     0\cdot \alpha +  1\cdot \alpha^2 +  0\cdot \alpha^4 \\ 
     0\cdot \alpha +  0\cdot \alpha^2 +  1\cdot \alpha^4 
\end{bmatrix}\\
             &=&\begin{bmatrix}
    \alpha   & \alpha^2 & \alpha^4\\
    \alpha^2 & \alpha^4 & \alpha\\
    \alpha^4 & \alpha & \alpha^2
\end{bmatrix}\cdot\begin{bmatrix}
   \alpha \\
     \alpha^2 \\ 
      \alpha^4 
\end{bmatrix}=\begin{bmatrix}
     1 \\
     0 \\ 
     0 
\end{bmatrix},\\
\end{eqnarray*}

so that $g_0 (x) =\alpha^5 x^4 +\alpha^5 x^2 +  \alpha^6 x$ and $g_1 (x) = x$. Therefore, $g(x):= g_0 (x) + g_1 \left(x^8 \right)=x^{8}+\alpha^5 x^4 +\alpha^5 x^2 +  \alpha^6 x$ is a PP in $\mathbb{F}_{64}$ with coefficients in $\mathbb{F}_8$ .\\

\end{example}

Since Theorem \ref{gi-coefficients} provides a way to compute the coefficients of the terms $g_{i}$'s that appear in the description of $\psi$  (in \ref{mainiso}), in terms of dual normal basis, we are ready to work with explicit computations understanding how the entries of a given matrix $A\in M_{m}\left(\mathcal{R}_{q,s}\right)$ affect the coefficients of the linear polynomial $\psi (A)\in R_m$ (as in Theorem  \ref{elem}).

\section{Permutation polynomials arising from elementary matrices in  $GL_m \left(\mathcal{R}_{q,s}\right)$}\label{S2}

In this section, three distinct methods to compute PPs in $R_{m}$ are presented. We start remembering the definition of the general linear group of degree $m$ over $\mathcal{R}_{q,s}$ and some of its most important subgroups. \\

 Let $U(\mathcal{R}_{q,s})$ denote the unity group of $\mathcal{R}_{q,s}$. The general linear group of degree $m$ over the ring $\mathcal{R}_{q,s}$, is the group    \[GL_{m}\left(\mathcal{R}_{q,s}\right):=\{A\in M_m \left(\mathcal{R}_{q,s}\right): A \text{ is invertible } \}=\{A\in M_m \left(\mathcal{R}_{q,s}\right): det(A) \in U(\mathcal{R}_{q,s}) \}.\] The special linear group of degree $m$ over the ring $\mathcal{R}_{q,s}$, is the subgroup of  $GL_m \left(\mathcal{R}_{q,s}\right)$ given by \[SL_{m}\left(\mathcal{R}_{q,s}\right)=\{A\in GL_m \left(\mathcal{R}_{q,s}\right): det(A)=1 \}.\] The Borel subgroup of  $GL_m \left(\mathcal{R}_{q,s}\right)$ is defined as \[B_{m}\left(\mathcal{R}_{q,s}\right)=\{A\in GL_m \left(\mathcal{R}_{q,s}\right):  A \text{ is an upper triangular matrix} \}.\]

Let $c(x) \in \mathcal{R}_{q,s}$ and $j,k\in \{1,...,m\}$ with $j\neq k$. The transvection $\chi_{jk}(c(x))\in M_{m}\left(\mathcal{R}_{q,s}\right)$ will be the matrix having $c(x)$ in the $jk$-entry, 1 in all diagonal entries, and $0$ elsewhere.  If $u(x)\in U(\mathcal{R}_{q,s})$, the row multiplication $D_{l}(u(x))\in M_{m}\left(\mathcal{R}_{q,s}\right)$ will be the matrix having $u(x)$ in the $ll$-entry, 1 in all the remaining diagonal entries, and $0$ elsewhere. Transvections and row multiplication matrices will be called elementary matrices (some authors use this term to refer to the group generated by transvections \cite{HOM, suslin}, while in linear algebra textbooks, this term is commonly used to refer to matrices obtained from the identity matrix by a single elementary row operation). Corollary \ref{csuper}  states that elementary matrices generate $ GL_{m}\left(\mathcal{R}_{q,s}\right)$, and transvections generate $SL_{m}\left(\mathcal{R}_{q,s}\right)$.

\begin{theorem}\label{super}
Let $R$ be a finite commutative ring with unity. Then :
\begin{enumerate}
     \item Transvections generate $SL_{m}\left(R\right)$

     \item Elementary matrices generate $ GL_{m}\left(R \right)$
\end{enumerate}    
\end{theorem}

\begin{proof}
It is well-known that, for commutative rings with unity, it is equivalent having a finite number of maximal ideals and being a semi-local ring. Since $R$ is finite, it has a finite number of maximal ideals and so $R$ is a semi-local ring. Hence, by  \cite[Theorem 4.3.9]{HOM}, transvections generate the group $SL_{m}\left(R\right)$ and the elementary matrices generate $ GL_{m}\left(R \right)$.   
\end{proof}

Theorem \ref{super} (part $1$) is not true in general, i.e., there exist rings $R$ such that $SL_{m}\left(R\right)$ is not generated by transvections (see, for example, \cite[p. 173]{HOM} or \cite{suslin} ).\\

\begin{corollary}\label{csuper}
The following statements hold true:
\begin{enumerate}
     \item Transvections generate $SL_{m}\left(\mathcal{R}_{q,s}\right)$

     \item Elementary matrices generate $ GL_{m}\left(\mathcal{R}_{q,s} \right)$
\end{enumerate} 
\end{corollary}

Since any invertible matrix in $GL_{m}(\mathcal{R}_{q,s})$ is the product of elementary matrices (by Corollary \ref{csuper}),  any PP in $R_{m}$ may be constructed as the image of this kind of product under $\psi$. To do so, one could solve the problem of  describing the image of elementary matrices under $\psi$ and then use the fact that it is multiplicative. Let $j,k\in \{1,...,m\}$ and $O_{jk}\in M_{m}(\mathcal{R}_{q,s})$ denote the matrix having $1$ in the $jk$-entry and $0$ elsewhere.  Since every  matrix in $M_{m}(\mathcal{R}_{q,s})$ is the sum of matrices in the set $\mathcal{O}:=\{O_{jk}c(x): c(x)\in \mathcal{R}_{q,s}  \wedge j,k\in \{1,...,m\} \}$, one could solve the problem of  describing the image of an arbitrary matrix in $M_{m}(\mathcal{R}_{q,s})$, by describing the image under $\psi$ of an element of the set $\mathcal{O}$ and then use the fact that it is additive.  In conclusion, since $\psi$ is an isomorphism of $\mathbb{F}_{q}$-algebras,  we have three alternatives for computing the image under $\psi$  of an element of $GL_{m}(\mathcal{R}_{q,s})$, as a ``product'' (remember the product in $R_{m}$ is the composition) of images of elementary matrices; as a sum of images of matrices in the set $\mathcal{O}$; or a combination of both. Theorem \ref{elem} will be of use to do that.\\

\begin{theorem} \label{elem}
 
  Let $\alpha \in \mathbb{F}_{q^{m}}$ be an $\mathbb{F}_{q}$-normal element, $B=\left\{\alpha, \alpha^{q},..., \alpha^{q^{m-1}}\right\}$ be the normal basis determined by $\alpha$, and $B'=\left\{u, u^{q},..., u^{q^{m-1}}\right \}$ the dual basis of $B$. Let $\displaystyle{c(x)=\sum_{i=0}^{s-1}c_{i}x^{i}\in \mathcal{R}_{q,s}}$. Then the following holds:
 \begin{enumerate}
\item  The set $\{O_{jk}x^{i}: j,k\in \{1,...,m\} \wedge  i\in \{0,...,s-1\}\}$ is an $\mathbb{F}_{q}$-basis of $M_{m}(\mathcal{R}_{q,s})$. In addition,

\[\psi\left(O_{jk}c(x)\right)= \sum_{i=0}^{s-1}\sum_{r=0}^{m-1}c_{i}\left(\alpha^{q^{j}}u^{q^{k+r}}\right)x^{q^{m\cdot i +r}}.\]
for all $j,k\in \{1,...,m\}$. In particular, if $B$ is a self-dual normal basis of $\mathbb{F}_{q^{m}}$ over $\mathbb{F}_{q}$, then

\[\psi\left(O_{jk}c(x)\right)= \sum_{i=0}^{s-1}\sum_{r=0}^{m-1}c_{i}\left(\alpha^{q^{j}+q^{k+r}}\right)x^{q^{m\cdot i +r}}.\]
for all $j,k\in \{1,...,m\}$.\\

\item  If $j,k\in \{1,...,m\}$ are distinct, then \[\psi\left(\chi_{jk}( c(x))\right)=x+ \sum_{i=0}^{s-1}\sum_{r=0}^{m-1}c_{i}\left(\alpha^{q^{j}}u^{q^{k+r}}\right)x^{q^{m\cdot i +r}}.\]
In particular, if $B$ is a self-dual normal basis of $\mathbb{F}_{q^{m}}$ over $\mathbb{F}_{q}$ , then

\[\psi\left(\chi_{jk}( c(x))\right)=x+ \sum_{i=0}^{s-1}\sum_{r=0}^{m-1}c_{i}\left(\alpha^{q^{j}+q^{k+r}}\right)x^{q^{m\cdot i +r}}.\]

\item If $\displaystyle{t(x)=\sum_{i=0}^{s-1}t_{i}x^{i}}\in \mathcal{R}_{q,s}$ is such that $\displaystyle{c(x)=t(x)+1\in U(\mathcal{R}_{q,s})}$,  and  $l\in \{1,...,m\}$, then \[\psi\left(D_{l}(c(x))\right)=x+ \sum_{i=0}^{s-1}\sum_{r=0}^{m-1}t_{i}\left(\alpha^{q^{l}}u^{q^{l+r}}\right)x^{q^{m\cdot i +r}}.\]
In particular, if $B$ is a self-dual normal basis of $\mathbb{F}_{q^{m}}$ over $\mathbb{F}_{q}$ , then

\[\psi\left(D_{l}(c(x))\right)=x+ \sum_{i=0}^{s-1}\sum_{r=0}^{m-1}t_{i}\left(\alpha^{q^{l}(1+q^{r})}\right)x^{q^{m\cdot i +r}}.\]

\end{enumerate}
\end{theorem}
\begin{proof}
\begin{enumerate}
\item Since the set $\{x^{i}: i\in \{1,...,s-1\}\}$ is an $\mathbb{F}_{q}$-basis of $\mathcal{R}_{q,s}$ and $\{O_{jk}: j,k\in \{1,...,m\}\}$ is an $\mathcal{R}_{q,s}$-basis of $M_{m}(\mathcal{R}_{q,s})$ (as free $\mathcal{R}_{q,s}$-module), then  $\{O_{jk}x^{i}: j,k\in \{1,...,m\} \wedge  i\in \{1,...,s-1\}\}$ is an $\mathbb{F}_{q}$-basis of $M_{m}(\mathcal{R}_{q,s})$. In addition, by definition of $\psi$, $\psi(O_{jk}x^{i})=g_{i}\left(x^{q^{m\cdot i}}\right)$ where $\displaystyle{g_{i}(x)=\sum_{r=0}^{m-1}a_{ir}x^{q^{r}}}$ and 

\[[g_{i}]_{B}=\left[[g_{i}(\alpha)]_{B}\left[g_{i}\left(\alpha^{q}\right)\right]_{B}\cdots \left[g_{i}\left(\alpha^{q^{m-1}}\right)\right]_{B}\right]=O_{jk}
\]
for $i=0,...,s-1$. Implying that the only non-zero entry of $[g_{i}]_{B}$, which is equal to $1$, is in the $j$-th position of the $k$-th column, i.e., the $j$-th position of  $[g_{i}(\alpha^{q^{k}})]_{B}$ is $1$ for $i=0,...,s-1$. Thus $g_{i}\left(\alpha^{q^{k}}\right)=\alpha^{q^{j}}$ and $g_{i}\left(\alpha^{q^{r}}\right)=0$ if $r\neq k$, for $i=0,...,s-1$. Hence, by Theorem \ref{gi-coefficients}
\begin{eqnarray*}
\begin{bmatrix}
a_{i0} \\ a_{i1} \\ \vdots \\ a_{im-1}
\end{bmatrix}&=& 
\begin{bmatrix}
u & u^{q} & \cdots & u^{q^{m-1}} \\
u^{q} & u^{q^{2}} & \cdots & u \\
\vdots  & \vdots  & \ddots & \vdots  \\
u^{q^{m-1}} &u & \cdots & u^{q^{m-2}} 
\end{bmatrix}
\begin{bmatrix}
g_{i}(\alpha) \\ g_{i}(\alpha^{q})\\ \vdots \\ g_{i}(\alpha^{q^{k}}) \\ \vdots \\ g_{i}(\alpha^{q^{m-1}}) 
\end{bmatrix}\\
             &=&\begin{bmatrix}
u & u^{q} & \cdots & u^{q^{m-1}} \\
u^{q} & u^{q^{2}} & \cdots & u \\
\vdots  & \vdots  & \ddots & \vdots  \\
u^{q^{m-1}} &u & \cdots & u^{q^{m-2}} 
\end{bmatrix}
\begin{bmatrix}
0 \\  \vdots \\ \alpha^{q^{j}} \\ \vdots \\0
\end{bmatrix}\\
            &=& \alpha^{q^{j}}\cdot \begin{bmatrix}
u^{q^{k}} \\ u^{q^{k+1}} \\ \vdots \\ u^{q^{k+(m-1)}} 
\end{bmatrix}= \begin{bmatrix}
\alpha^{q^{j}}u^{q^{k}} \\\alpha^{q^{j}} u^{q^{k+1}} \\ \vdots \\ \alpha^{q^{j}}u^{q^{k+(m-1)}} 
\end{bmatrix}
\end{eqnarray*} 
for $i=0,...,s-1$. Hence $\displaystyle{g_{i}(x)=\sum_{r=0}^{m-1}a_{ir}x^{q^{r}}=\sum_{r=0}^{m-1}\left(\alpha^{q^{j}}u^{q^{k+r}}\right)x^{q^{r}}}$ and so 
\begin{eqnarray*}
\psi\left(O_{jk}x^{i}\right)&=&g_{i}\left(x^{q^{m\cdot i}}\right)=\sum_{r=0}^{m-1}\left(\alpha^{q^{j}}u^{q^{k+r}}\right)\left(x^{q^{m\cdot i}}\right)^{q^{r}}                                                    \\
                            &=&\sum_{r=0}^{m-1}\left(\alpha^{q^{j}}u^{q^{k+r}}\right)x^{q^{m\cdot i + r}}.
\end{eqnarray*}   

Thus, since
\begin{equation*}
O_{jk}c(x)=c_{0}O_{jk}+c_{1}O_{jk}x+\cdots + c_{s-1}O_{jk}x^{s-1}
\end{equation*}

 then 
\begin{eqnarray*}
\psi(O_{jk}c(x))&=&\psi(c_{0}O_{jk}+c_{1}O_{jk}x+\cdots + c_{s-1}O_{jk}x^{s-1})\nonumber\\
                             &=&c_{0}\psi(O_{jk})+c_{1}\psi(O_{jk}x)+\cdots + c_{s-1}\psi(O_{jk}x^{s-1}) \\
                             &=& \sum_{i=0}^{s-1} c_{i} \left( \psi(O_{jk}x^{i})\right)
                             =  \sum_{i=0}^{s-1} c_{i} \left( \sum_{r=0}^{m-1}\left(\alpha^{q^{j}}u^{q^{k+r}}\right)x^{q^{m\cdot i +r}}\right) \\
                             &=& \sum_{i=0}^{s-1}\sum_{r=0}^{m-1}c_{i}\left(\alpha^{q^{j}}u^{q^{k+r}}\right)x^{q^{m\cdot i +r}},  
\end{eqnarray*}

In particular,  if $B$ is a self-dual normal basis of $\mathbb{F}_{q^{m}}$ over $\mathbb{F}_{q}$, $\alpha=u$ and so \[\psi(c(x)O_{jk})=\sum_{i=0}^{s-1}\sum_{r=0}^{m-1}c_{i}\left(\alpha^{q^{j}+q^{k+r}}\right)x^{q^{m\cdot i +r}}.\]

\item Let $j,k\in \{1,...,m\}$ distinct and $I_{m}\in M_{m}(\mathcal{R}_{q,s})$ denote the identity matrix. Since $\chi_{jk}(c(x))=I_{m} + O_{jk}c(x)$, then 

\begin{eqnarray*}
\psi(\chi_{jk}(c(x)))&=&\psi(I_{m} + O_{jk}c(x))
                             =\psi(I_{m}) + \psi(O_{jk}c(x))\\
                     &=& x+ \sum_{i=0}^{s-1}\sum_{r=0}^{m-1}c_{i}\left(\alpha^{q^{j}}u^{q^{k+r}}\right)x^{q^{m\cdot i +r}}
\end{eqnarray*}

where the last equality is by part $1$. In particular, if $B$ is a self-dual normal basis of $\mathbb{F}_{q^{m}}$ over $\mathbb{F}_{q}$, $\alpha=u$ so that
\[\
\psi(\chi_{jk}(c(x)))=x+ \sum_{i=0}^{s-1}\sum_{r=0}^{m-1}c_{i}\left(\alpha^{q^{j}+q^{k+r}}\right)x^{q^{m\cdot i +r}}           
\]

\item  Let $\displaystyle{t(x)=\sum_{i=0}^{s-1}t_{i}x^{i}}\in \mathcal{R}_{q,s}$ is such that $\displaystyle{c(x)=t(x)+1\in U(\mathcal{R}_{q,s})}$,  and  $l\in \{1,...,m\}$. Since $D_{l}(c(x))=I_{m} + O_{ll}t(x)$, then

\begin{eqnarray*}
\psi \left(D_{l}(c(x))\right)&=& \psi \left(I_{m} + O_{ll}t(x)\right) = \psi(I_{m}) + \psi(O_{ll}t(x))\\
                                    &=&x+ \sum_{i=0}^{s-1}\sum_{r=0}^{m-1}t_{i}\left(\alpha^{q^{l}}u^{q^{l+r}}\right)x^{q^{m\cdot i +r}}
\end{eqnarray*}
 where the last equality is by part $1$. In particular, if $B$ is a self-dual normal basis of $\mathbb{F}_{q^{m}}$ over $\mathbb{F}_{q}$, $\alpha=u$ and so

\[\psi\left(D_{l}(c(x))\right)=x+ \sum_{i=0}^{s-1}\sum_{r=0}^{m-1}t_{i}\left(\alpha^{q^{l}(1+q^{r})}\right)x^{q^{m\cdot i +r}}.\]

\end{enumerate}  
\end{proof}

\section{Families of permutation polynomials in $R_{m}$}\label{S3}

In this section, we introduce and characterize several families of linear permutations in $R_{m}$. A Special permutation polynomial (or SPP) will be an element of $\psi(SL_{m}\left(\mathcal{R}_{q,s}\right))$. A Borel permutation polynomial (or BPP) will be an element of $\psi(B_{m}\left(\mathcal{R}_{q,s}\right))$, and a special Borel permutation polynomial (or SBPP) will be an element of $\psi(SL_{m}\left(\mathcal{R}_{q,s}\right) \cap B_{m}\left(\mathcal{R}_{q,s}\right) )$.  Let $D_{m}(\mathcal{R}_{q,s})$ denote the subgroup of $GL_{m}(\mathcal{R}_{q,s})$ formed by the diagonal invertible matrices.  A diagonal permutation polynomial (DPP)  will be an element of $\psi(D_{m}\left(\mathcal{R}_{q,s}\right))$, and a special diagonal permutation polynomial (or SDPP) will be an element of $\psi(SL_{m}\left(\mathcal{R}_{q,s}\right) \cap D_{m}\left(\mathcal{R}_{q,s}\right) )$. Theorems \ref{celem} and \ref{PP} give  characterizations of PPs, SPPs, BPPs, SBPPs, DPPs, and SDPPs.\\

\begin{theorem}\label{celem}

Let  $f(x)\in R_{m}$. Then:

\begin{enumerate}

\item $f(x)$ is a SPP iff there exists $t\in \mathbb{Z}_{>0}$ and 

 $j_{d}, k_{d}\in \{1,...,m\}$ with $j_{d}\neq k_{d}$ for all $d\in \{1,...,t\}$ such that 

\[f(x)=\bigcirc_{d=1}^{t}\left(x+ \sum_{i=0}^{s-1}\sum_{r=0}^{m-1}c_{di}\left(\alpha^{q^{j_{d}}}u^{q^{k_{d}+r}}\right)x^{q^{m\cdot i +r}} \right)\]

where $\bigcirc$ denotes the composition of functions and $c_{di}$ is an arbitrary element in $\mathbb{F}_{q}$ for all $d\in \{1,...,t\}$ and $i\in \{1,...,s-1\}$.\\ 
%

\item  $f(x)$ is a PP iff there exists an SPP $h(x)$, a polynomial $t(x)=\sum_{i=0}^{s-1}t_{i}x^{i}\in \mathcal{R}_{q,s}$ such that $t(x)+1 \in U\left(\mathcal{R}_{q,s}\right)$, and

\[f(x)=h(x)+ \sum_{i=0}^{s-1}\sum_{r=0}^{m-1}t_{i}\left(\alpha^{q}u^{q^{r}}\right)(h(x))^{q^{m\cdot i +r}} .\]

%

\end{enumerate}

\end{theorem}

\begin{proof}

\begin{enumerate}

\item If $f(x)\in \psi(SL_{m}\left(\mathcal{R}_{q,s}\right))$ then there exists $t\in \mathbb{Z}_{>0}$  and 

\begin{eqnarray*}
f(x)&=&\psi\left(\prod_{d=1}^{t} \chi_{j_{d}k_{d}}(c_{d}(x))\right)\\
    &=&\bigcirc_{d=1}^{t}\psi\left(\chi_{j_{d}k_{d}}(c_{d}(x))\right)\\
    &=&\bigcirc_{d=1}^{t}\left(x+ \sum_{i=0}^{s-1}\sum_{r=0}^{m-1}c_{di}\left(\alpha^{q^{j_{d}}}u^{q^{k_{d}+r}}\right)x^{q^{m\cdot i +r}} \right)
\end{eqnarray*}

 where $j_{d}, k_{d}\in \{1,...,m\}$, $j_{d}\neq k_{d}$, and  $c_{d}(x)=\sum_{i=0}^{s-1}c_{di}x^{i}\in \mathcal{R}_{q,s}$ for all $d\in \{1,...,t\}$; the first equality is by Corollary \ref{csuper} (part $1$); the second  and the third ones are because $\psi$ preserves products and by Theorem \ref{elem} (part $2$), respectively. The converse is clear.\\

\item Note that, if $A\in GL_{m}(\mathcal{R}_{q,s})$, then $A=D_{1}(det(A))B$ where $B:=D_{1}(det(A)^{-1})A$ belongs to $SL_{m}(\mathcal{R}_{q,s})$. Thus, if $f(x)\in \psi(GL_{m}\left(\mathcal{R}_{q,s}\right))$, there exists $A\in GL_{m}(\mathcal{R}_{q,s})$ such that

\begin{eqnarray*}
f(x)&=&\psi\left(D_{1}(det(A)) \cdot \left[D_{1}(det(A)^{-1})A\right]\right)\\
    &=&\psi\left(D_{1}(det(A)) \right) \circ \psi\left(D_{1}(det(A)^{-1})A\right)\\
    &=& \left( x+ \sum_{i=0}^{s-1}\sum_{r=0}^{m-1}t_{i}\left(\alpha^{q}u^{q^{r}}\right)(x)^{q^{m\cdot i +r}}\right) \circ h(x)\\
    &=&h(x)+ \sum_{i=0}^{s-1}\sum_{r=0}^{m-1}t_{i}\left(\alpha^{q}u^{q^{r}}\right)(h(x))^{q^{m\cdot i +r}}
\end{eqnarray*} 
where $h(x):=\psi\left(D_{1}(det(A)^{-1})A\right)$ is a SPP and $t(x)=\sum_{i=0}^{s-1}t_{i}x^{i}:=det(A)-1$ is such that $t(x)+1\in U(\mathcal{R}_{q,s})$. Conversely, suppose there exists an SPP $h(x)$, and a polynomial $t(x)=\sum_{i=0}^{s-1}t_{i}x^{i}\in \mathcal{R}_{q,s}$ such that $t(x)+1 \in U\left(\mathcal{R}_{q,s}\right)$, and

\[f(x)=h(x)+ \sum_{i=0}^{s-1}\sum_{r=0}^{m-1}t_{i}\left(\alpha^{q}u^{q^{r}}\right)(h(x))^{q^{m\cdot i +r}} .\]
Thus, since $\psi(D_{1}(t(x)+1))$ and $h(x)$ are both PPs, then $f(x)=\psi(D_{1}(t(x)+1))\circ h(x)$ is a PP.

\end{enumerate}

\end{proof}

Note that, in spite of the depiction given in Theorem \ref{celem} being useful to construct any SPP or PP, it did not provide an explicit description of the coefficients of these. Nonetheless, the other families of PPs we introduced can be characterized by conditions on their coefficients (as seen in Theorem \ref{PP}).

\begin{theorem}\label{PP}  Let $\alpha \in \mathbb{F}_{q^{m}}$ be an $\mathbb{F}_{q}$-normal element, $B=\left\{\alpha, \alpha^{q},..., \alpha^{q^{m-1}}\right\}$ be the normal basis determined by $\alpha$, and $B'=\left\{u, u^{q},..., u^{q^{m-1}}\right \}$ the dual basis of $B$. Let  $f(x)\in R_{m}$. Then:


\begin{enumerate}

\item $f(x)$ is a BPP iff  there exists a collection $\{c_{jj}(x)=\sum_{i=0}^{s-1}c_{jji}x^{i}\in \mathcal{R}_{q,s}: $  $ j=1,...,m \}$ such that $gcd(\prod_{j=1}^{m}c_{jj}(x), x^{s}-1)=1$ and

\[f(x)=\sum_{1\leq j\leq k\leq m}\sum_{i=0}^{s-1}\sum_{r=0}^{m-1}c_{jki}(\alpha^{q^{j}}u^{q^{k+r}})x^{q^{m\cdot i +r}}\]

 where the coefficients $c_{jki}$ are arbitrary elements in $\mathbb{F}_{q}$
if $j<k$, for all $i\in \{0,...,s-1\}$.\\

\item $f(x)$ is an SBPP iff  there exists a collection $\{c_{jj}(x)=\sum_{i=0}^{s-1}c_{jji}x^{i}\in \mathcal{R}_{q,s}:$ $  j=1,...,m \}$ such that $\prod_{j=1}^{m}c_{jj}(x)=1$  and

\[f(x)=\sum_{1\leq j\leq k\leq m}\sum_{i=0}^{s-1}\sum_{r=0}^{m-1}c_{jki}(\alpha^{q^{j}}u^{q^{k+r}})x^{q^{m\cdot i +r}}\]

 where the coefficients $c_{jki}$ are arbitrary elements in $\mathbb{F}_{q}$
if $j<k$, for all $i\in \{0,...,s-1\}$. In particular, any polynomial of the form 

\[f(x)=\sum_{j=1}^{m}\sum_{r=0}^{m-1}(\alpha^{q^{j}}u^{q^{j+r}})x^{q^{r}} + \sum_{1\leq j< k\leq m}\sum_{i=0}^{s-1}\sum_{r=0}^{m-1}c_{jki}(\alpha^{q^{j}}u^{q^{k+r}})x^{q^{m\cdot i +r}},\]

where the coefficients $c_{jki}$ are arbitrary elements in $\mathbb{F}_{q}$ for all $i\in \{0,...,s-1\}$, is an SBPP.\\

\item  $f(x)$ is a DPP iff  there exists a collection $\{c_{j}(x)=\sum_{i=0}^{s-1}c_{ji}x^{i}\in \mathcal{R}_{q,s}:$ $  j=1,...,m\}$ such that  $gcd(\prod_{j=1}^{m}c_{j}(x), x^{s}-1)=1$  and

\[f(x)=\sum_{j=1}^{m}\sum_{i=0}^{s-1}\sum_{r=0}^{m-1}c_{ji}(\alpha^{q^{j}}u^{q^{j+r}})x^{q^{m\cdot i +r}}.\]


\item $f(x)$ is an SDPP iff  there exists a collection $\{c_{j}(x)=\sum_{i=0}^{s-1}c_{ji}x^{i}\in \mathcal{R}_{q,s}:$ $  j=1,...,m\}$ such that $\prod_{j=1}^{m}c_{j}(x)=1$ and

\[f(x)=\sum_{j=1}^{m}\sum_{i=0}^{s-1}\sum_{r=0}^{m-1}c_{ji}(\alpha^{q^{j}}u^{q^{j+r}})x^{q^{m\cdot i +r}}.\]

\end{enumerate}

\end{theorem}

\begin{proof}

\begin{enumerate}

\item Let  $f(x)\in \psi (A)$ with $A\in B_{m}(\mathcal{R}_{q,s})$. Note that $A=\sum_{1\leq j\leq k\leq m}$ $c_{jk}(x)O_{jk}$ where $c_{jk}(x)=\sum_{i=0}^{s-1}c_{jki}x^{i}\in \mathcal{R}_{q,s}$ is arbitrary if $j<k$, and $det(A)=\prod_{j=1}^{m}c_{jj}(x)\in U(\mathcal{R}_{q,s})$  (i.e., $gcd(\prod_{j=1}^{m}c_{jj}(x), x^{s}-1)=1$). Thus there exists a collection $\{c_{jj}(x)=\sum_{i=0}^{s-1}c_{jji}x^{i}\in \mathcal{R}_{q,s}: $  $ j=1,...,m \}$ such that $gcd(\prod_{j=1}^{m}c_{jj}(x), x^{s}-1)=1$ and

\begin{eqnarray*}
f(x)&=&\psi\left(\sum_{1\leq j\leq k\leq m}c_{jk}(x)O_{jk}\right)=\sum_{1\leq j\leq k\leq m}\psi\left(c_{jk}(x)O_{jk}\right)\\
    &=&\sum_{1\leq j\leq k\leq m}\sum_{i=0}^{s-1}\sum_{r=0}^{m-1}c_{jki}(\alpha^{q^{j}}u^{q^{k+r}})x^{q^{m\cdot i +r}}
\end{eqnarray*}

 where the coefficients $c_{jki}$ are arbitrary elements in $\mathbb{F}_{q}$
if $j<k$, for all $i\in \{0,...,s-1\}$; and the last equality is by Theorem \ref{elem} (part $1$).\\

\item It follows form part $1$.\\

\item Let  $f(x)\in \psi (A)$ with $A\in D_{m}(\mathcal{R}_{q,s})$. Note that $A=\sum_{j=1}^{m}c_{j}(x)O_{jj}$ where $c_{j}(x)=\sum_{i=0}^{s-1}c_{ji}x^{i}\in \mathcal{R}_{q,s}$  and $det(A)=\prod_{j=1}^{m}c_{j}(x)\in U(\mathcal{R}_{q,s})$  (i.e., $gcd(\prod_{j=1}^{m}c_{j}(x), x^{s}-1)=1$). Thus there exists a collection $\{c_{j}(x)=\sum_{i=0}^{s-1}c_{ji}x^{i}\in \mathcal{R}_{q,s}:$ $  j=1,...,m\}$ such that $gcd(\prod_{j=1}^{m}c_{j}(x), x^{s}-1)=1$ and

\begin{eqnarray*}
f(x)&=&\psi\left(\sum_{j=1}^{m}c_{j}(x)O_{jj}\right)=\sum_{j=1}^{m}\psi\left(c_{j}(x)O_{jj}\right)\\
    &=&\sum_{j=1}^{m}\sum_{i=0}^{s-1}\sum_{r=0}^{m-1}c_{ji}(\alpha^{q^{j}}u^{q^{j+r}})x^{q^{m\cdot i +r}}
\end{eqnarray*}

 where the last equality is by Theorem \ref{elem} (part $1$).\\

\item It follows form part $3$.\\

\end{enumerate}

\end{proof}

\begin{example} Let $q=3$, $m=3$, $s=6$, then $n=ms=18$ and $\mathbb{F}_{q^{n}}=\mathbb{F}_{3^{18}}$.
In this example, we present a collection of BPPs in $R_{m}=R_{3}$. Let $\displaystyle{\mathbb{F}_{q^m}=\mathbb{F}_{3^3}\cong \frac{\mathbb{F}_3 [x]}{\left\langle x^3 +2x+1\right\rangle}}$, and $\gamma\in \mathbb{F}_{3^3}$ be a primitive element. It is easy to check that $B=\left\{\alpha=\gamma^2,\alpha^{3}=\gamma^6 , \alpha^{9}=\gamma^{18} \right\}$ is a normal $\mathbb{F}_{3}$-basis of $\mathbb{F}_{3^3}$ with dual normal basis $B'=\left\{u=\gamma^{21}, u^{3}=\gamma^{11}, u^{9}=\gamma^7\right\}$. Let
\[
G=\begin{bmatrix}
x^2 +x +2 & c_{12}(x) & c_{13}(x)\\
0 & x^2 +1 & c_{23}(x)\\
0 & 0 & x^4 +1
\end{bmatrix}\in M_3\left(\mathcal{R}_{3,6} \right),
\]
then $\gcd\left(\left(x^2 +x +2\right)\left(x^2 +1\right)\left(x^4 +1\right), x^6 -1\right)=1$, i.e., $G \in B_3 \left(\mathcal{R}_{3,6}\right)$.
Then, by Theorem \ref{PP}(part $2$), 

\begin{eqnarray*}
\psi(G)&=&\sum_{1\leq j\leq k\leq m}\sum_{i=0}^{s-1}\sum_{r=0}^{m-1}c_{jki}\left(\alpha^{q^{j}}u^{q^{k+r}}\right)x^{q^{m\cdot i +r}}\\
       &=& \sum_{1\leq j\leq k\leq 3}\sum_{i=0}^{5}\sum_{r=0}^{2}c_{jki}\left(\alpha^{3^{j}}u^{3^{k+r}}\right)x^{3^{3\cdot i +r}}\\
&=&\sum_{j=1}^{3}\sum_{i=0}^{5}\sum_{r=0}^{2}\left(\alpha^{3^{j}}u^{3^{j+r}}\right)x^{3^{m\cdot i +r}}\\ 
       &+&\sum_{1\leq j< k\leq 3}\sum_{i=0}^{5}\sum_{r=0}^{2}c_{jki}\left(\alpha^{3^{j}}u^{3^{k+r}}\right)x^{3^{m\cdot i +r}}\nonumber\\
&=&g(x)+ \sum_{1\leq j< k\leq 3}\sum_{i=0}^{5}\sum_{r=0}^{2}c_{jki}\left(\alpha^{3^{j}}u^{3^{k+r}}\right)x^{3^{m\cdot i +r}},
\end{eqnarray*}
%
%
where $g(x)= \displaystyle{\sum_{j=1}^{3}\sum_{i=0}^{5}\sum_{r=0}^{2}\left(\alpha^{3^{j}}u^{3^{j+r}}\right)x^{3^{m\cdot i +r}}}=
\gamma^{13}x^{4782969} +
\gamma^{18}x^{1594323}+
\gamma^{23}x^{531441}+
\gamma^{22} x^{6561}+
\gamma x^{2187}+
\gamma^{6} x^{729}+
\gamma^7 x^{243}+
\gamma^{12}x^{81}+\gamma^{17}x^{27}+
\gamma^{10} x^9 +\gamma^{15} x^3+
\gamma^{20}x$, is a BPP. 
\end{example}

\section{Sizes of families of permutation polynomials in $R_{m}$}\label{S4}

In this section, the sizes of the sets of SPPs, BPPs, and DPPs are computed. Computing the size of remarkable subgroups of the general linear group $GL_{m}(R)$  of degree $m$  over a finite ring $R$ (i.e., the group of invertible $m\times m$ matrices over $R$) is a  problem of interest from a combinatorial point of view. A formula for the size of $GL_{m}(\mathcal{R}_{q,s})$ was computed in \cite[Theorem 4.3]{bcv}, while a general formula for the size of $GL_{m}(R)$, for an arbitrary finite ring $R$, was computed later in  \cite[Corollary 3.2]{ko}.\\

 In Theorem \ref{sizes}, we are going to present a formula to compute the size of the groups $SL_{m}\left(\mathcal{R}_{q,s}\right)$, $B_{m}\left(\mathcal{R}_{q,s}\right)$, and $D_{m}(\mathcal{R}_{q,s})$, i.e., the size of the sets of SPPs, BPPs, and DPPs, respectively. The size of $SL_{m}(\mathcal{R}_{q,s})$ is obtained as an application of Theorem \ref{bcv} (this could be done, alternatively, using  \cite[Corollary 3.2]{ko}). 

\begin{theorem}\cite[Theorem 4.3]{bcv} \label{bcv}
Let $\displaystyle{x^{s} -1 = \prod_{j=1}^t f_{j}^{s_{j}} }$, where the $f_{j}$'s are distinct irreducible elements of $\mathbb{F}_q [x]$,  $d_{j}= deg(f_{j})$ and $\displaystyle{A_{j} =\frac{\mathbb{F}_{q} [x]}{\left\langle f_{j}^{s_{j}}\right\rangle}}$ for $j=1,2,\ldots, t.$. Then $GL_{m} \left(\mathcal{R}_{q,s}\right)$ is isomorphic to the direct product (of groups) $GL_{m}(A_{1})\times \cdots \times GL_{m}(A_{t}) $. Moreover
\begin{equation*}
    \left|GL_m \left(\mathcal{R}_{q,s}\right)\right| = q^{m^{2}s} \prod_{j=1}^t \prod_{i=1}^m \left(1 - q^{-i\cdot d_j}\right).
\end{equation*}
\end{theorem}

If $s=1$ in Theorem \ref{bcv}, one get the well-known formula 
\begin{eqnarray*}
|GL_{n}(\mathbb{F}_{q})|&=& q^{{n}^2}  \prod_{i=1}^n \left(1 - q^{-i}\right)=  \prod_{i=1}^n q^{n}\left(1 - q^{-i}\right)\\
                        &=& (q^{n}-q^{n-1})\cdot (q^{n}-q^{n-2})\cdots (q^{n}-1).
\end{eqnarray*}

\begin{theorem}\label{sizes}
 Let $x^{s}-1=\prod_{j=1}^{t}f_{j}^{s_{j}}$ be the decomposition of $x^{s}-1$ (over $\mathbb{F}_{q}[x]$) into powers of irreducible distinct factors, and $d_{j}=deg(f_{j})$ for $j=1,...,t$. Then the following holds:

 \begin{enumerate}
     \item $\displaystyle{\left| U\left(\mathcal{R}_{q,s}\right)\right| = \prod_{j=1} ^t \left(q^{d_j}-1 \right) q^{d_j\left(s_j - 1\right)}}$
     
     \item $\displaystyle{\left|SL_m\left(\mathcal{R}_{q,s}\right) \right|=\frac{q^{m^2s} \prod_{j=1} ^t \prod_{i=1} ^m \left(1-q^{-i\cdot d_j}\right)}{\prod_{j=1}^t \left(q^{d_j} -1 \right) q^{d_j\left(s_j -1\right)}}}$
     \item $\displaystyle{\left|B_{m} \left(\mathcal{R}_{q,s}\right)\right|=q^{s{\left(\frac{m(m-1)}{2}\right)}}\left(\prod_{j=1} ^t \left(q^{d_j}-1\right)q^{d_j\left(s_j -1\right)}\right)}^{m}$
     
     \item $\displaystyle{ \left|D_{m} \left(\mathcal{R}_{q,s}\right)\right|=\left(\prod_{j=1} ^t \left(q^{d_j}-1\right)q^{d_j\left(s_j -1\right)}\right)}^{m}$
     
 \end{enumerate}

\end{theorem}

\begin{proof}

\begin{enumerate}
   \item Note that
    
\begin{equation*}
\left| U\left(\mathcal{R}_{q,s}\right)\right|=\left| U\left(\frac{\mathbb{F}_q [x]}{\left\langle x^s -1 \right\rangle} \right) \right|=\left|U\left(\prod_{j=1} ^t \frac{\mathbb{F}_q [x]}{\left\langle f_j ^{s_j} \right\rangle}\right)\right|=\left|\prod_{j=1} ^t  U\left(\frac{\mathbb{F}_q [x]}{\left\langle f_j ^{s_j} \right\rangle}\right)\right|,
\end{equation*}

where the second equality is by the Chinese Remainder Theorem. Let $\displaystyle{A_{j}:= \frac{\mathbb{F}_q [x]}{\left\langle f_{j} ^{s_{j}} \right\rangle}}$, then $A_{j}$ is a finite chain ring with maximal ideal $M_{j}:=A_{j}\overline{f_{j}}$ for $j=1,...,t$. Thus $\displaystyle{ M_{j}=\frac{\mathbb{F}_{q} [x]  f_{j}}{ \left\langle f_{j}^{s_{j}} \right\rangle }}$ so that
\[
  \left|\frac{A_{j}}{M_{j}}\right| =\left|\frac{\frac{\mathbb{F}_q [x]}{\left\langle f_j ^{s_{j}}\right\rangle}}{\frac{\mathbb{F}_q [x] f_{j} }{\left\langle f_{j} ^{s_{j}}\right\rangle}}\right| 
  =\left|\frac{\mathbb{F}_{q} [x]}{\mathbb{F}_{q} [x] f_{j}}\right|=q^{d_{j}},
\]
where the second equality is by the Third Isomorphism Theorem. Thus, by \cite[Lemma $2.1$, part $5$]{choosuwan}, $|U\left(A_{j}\right)|=\left(q^{d_{j}}-1 \right) q^{d_{j}\left(s_{j} - 1\right)}$ for $j=1,...,t$. Hence  $\displaystyle{\left| U\left(\mathcal{R}_{q,s}\right)\right| = \prod_{j=1} ^t \left(q^{d_j}-1 \right) q^{d_j\left(s_j - 1\right)}}$.\\

    \item  Since
$$
\begin{array}{cccc}
   det: & GL_m \left( \mathcal{R}_{q,s}\right) & \rightarrow & U\left(\mathcal{R}_{q,s}\right) \\
     &A&\mapsto&det(A) 
\end{array}
$$ 
is a group epimorphism and $Ker(det)=SL_m \left(\mathcal{R}_{q,s}\right)$, then $\left| U\left(\mathcal{R}_{q,s}\right)\right| = $ $\dfrac{\left|GL_m \left(\mathcal{R}_{q,s}\right)\right|}{\left|SL_m \left(\mathcal{R}_{q,s}\right)\right|}$ (by the First Isomorphism Theorem). Thus, by Theorem \ref{bcv} and part $1$,
\[
    \left|SL_{m} \left(\mathcal{R}_{q,s}\right)\right| = \frac{\left|GL_{m} \left(\mathcal{R}_{q,s}\right)\right|}{\left| U\left(\mathcal{R}_{q,s}\right)\right|} = \frac{q^{m^{2}s} \prod_{j=1}^t \prod_{i=1}^m \left(1-q^{-i\cdot d_j}\right)}{\prod_{j=1}^t \left(q^{d_{j}} -1 \right) q^{d_{j}\left(s_{j} -1\right)}}.
\]

  \item Since the determinant of any matrix in $B_{m}(\mathcal{R}_{q,s})$ is a unity obtained as the product of the elements in the main diagonal, all these elements must be unities. Hence, there are $\left| U\left(\mathcal{R}_{q,s}\right)\right|$ possibilities for each one of the $m$ entries in the main diagonal of any matrix in $B_{m}(\mathcal{R}_{q,s})$.  On the other hand, there are $|\mathcal{R}_{q,s}|=q^{s}$ possibilities for each one of the $\displaystyle{\frac{m(m-1)}{2}}$ entries over the main diagonal of any matrix in $B_{m}(\mathcal{R}_{q,s})$. Hence 

\begin{eqnarray*}
|B_{m}(\mathcal{R}_{q,s})|&=&\left(q^{s}\right)^{\left(\frac{m (m-1)}{2}\right)}\left| U\left(\mathcal{R}_{q,s}\right)\right|^{m}\\ 
                          &=&  q^{s\left(\frac{m(m-1)}{2}\right)}\left(\prod_{j=1}^t \left(q^{d_j} -1 \right) q^{d_j\left(s_j -1\right)}\right)^{m}
\end{eqnarray*}

where the last equality is by part $1$.\\

\item Since the determinant of any matrix in $D_{m}(\mathcal{R}_{q,s})$ is a unity obtained as the product of the elements in the main diagonal, all these elements must be in $U\left(\mathcal{R}_{q,s}\right)$.  Hence $ \left|D_{m}(\mathcal{R}_{q,s})\right|= \left|U\left(\mathcal{R}_{q,s}\right)\right|^{m}$.

\end{enumerate}
\end{proof}


\section{Linear permutations polynomials over $\mathbb{F}_{q^{2p}}$ with prescribed coefficients in $\mathbb{F}_{q^{2}}$}\label{S5}

In this section, some linear PPs of $\mathbb{F}_{q^{2p}}$ (where $p$ is a prime) whose coefficients are prescribed and belong to the subfield $\mathbb{F}_{q^{2}}$ are studied. In this case,   $s=p$ and $m=2$ so that the algebra $R_{2}$ of linear polynomials of $\mathbb{F}_{q^{2p}}$ with coefficient in $\mathbb{F}_{q^{2}}$ is isomorphic to $M_{2}(\mathcal{R}_{q,p})$ (by Theorem \ref{thbcv}).\\

\begin{lemma}\label{corgusta}
Let $\mathbb{F}_q$ be a finite field and $p$ an odd prime  such that $q$ is a primitive element in $\mathbb{F}_{p}$. Given a linear permutation polynomial $\displaystyle{f(x)=\sum_{j=0}^{p -1} f_j x^{q^j}}$ over $\mathbb{F}_{q}$, if
the following conditions hold true:
\begin{enumerate}
\item $\displaystyle{\sum_{j=0}^{p -1}f_j\neq 0}$
\item  $\displaystyle{-\dfrac{1}{p}\sum_{j=1}^{p -1} f_{j}+\left(1-\dfrac{1}{p}\right)f_0\neq 0,}$
\end{enumerate}
over $\mathbb{F}_q$, then $f(x)$ is a linear permutation over $\mathbb{F}_{q^{2p}}$.
\end{lemma}

\begin{proof}
 It follows from taking $p=2$ and $m=1$ in \cite[Corollary 4.2]{gustavo}.   
\end{proof}

\begin{theorem}\label{coefpol}
 Let  $\alpha \in \mathbb{F}_{q^{2}}$ be an $\mathbb{F}_{q}$-normal element, $B=\left\{\alpha, \alpha^{q}\right\}$ be the normal basis determined by $\alpha$, and $B'=\left\{u, u^{q}\right \}$ the dual basis of $B$. Then the following holds:
 \begin{enumerate}
     \item  If $g(x)\in R_{2}$ then there exist coefficients $f_{11i}, f_{12i}, f_{21i}, f_{22i}\in \mathbb{F}_{q}$ for $i=1,...,p-1$, such that $\displaystyle{g(x)=\sum_{i=0} ^{p -1} g_i \left(x^{q^{2i}}\right)}$ where 
     \[g_i (x)=\left[ \left(f_{12i} \alpha^{q} + f_{22i} \alpha \right)u + \left(f_{11i} \alpha^{q} + f_{21i} \alpha \right)u^q  \right]x + \left[\left(f_{11i} \alpha^{q} + f_{21i} \alpha \right)u + \left(f_{12i} \alpha^{q} + f_{22i} \alpha\right)u^{q}  \right]x^q ,\]   
  for all $i=1,...,s-1$.
  \item Let $g(x)\in R_{2}$ be represented as in part $1$. Then $g(x)$ is a PP iff 
  \[gcd\left(\sum_{i=0} ^{p -1}\left(\sum_{i\equiv r+s \mod p } f_{11r}f_{22s} - f_{12r} f_{21s}  \right)x^i, x^{p}-1 \right)=1.\]
  In particular, if $p$ is an odd prime such that $q$ is a primitive element in $\mathbb{F}_{p}$ and
\[\sum_{i=0} ^{p -1}\left(\sum_{i\equiv r+s \mod p } f_{11r}f_{22s} - f_{12r} f_{21s}  \right)\not\in \left\{0,\, p\sum_{{\substack{0\equiv r+s\mod p \\ 0\leq r \leq  s\leq p-1}}}\left(f_{11r}f_{22s} - f_{12r} f_{21s}\right)\right\},\]
then $g(x)$ is a PP.
 \end{enumerate}
 \end{theorem}

\begin{proof}
\begin{enumerate}
    \item Let $g(x)\in R_{2}$, then there exists a matrix $G\in M_{2}(\mathcal{R}_{q,p})$ such that $\psi(G)=g(x)$ (by Theorem \ref{thbcv}). Since any $G\in M_{2}(\mathcal{R}_{q,p})$, it can be written as elements in $\mathcal{O}=\{O_{jk}c(x): c(x)\in \mathcal{R}_{q,p}  \wedge j,k\in \{1,2\} \}$, then $G=\sum_{1\leq j,k\leq 2}\left(\sum_{i=0}^{p-1}f_{jki}x^{i}\right)O_{jk}$ where $\sum_{i=0}^{p-1}f_{jki}x^{i}\in \mathcal{R}_{q,p}$ for $i=1,...,p-1$. Thus, using Theorem \ref{elem} (part $1$) one gets the desired representation for $g(x)$.
    \item $g(x)$ is a PP iff $G:=\varphi(g(x))$ is an invertible matrix; or equivalently, $det(G)$ is a unity in $\mathcal{R}_{q,p}$; which means, $gcd(det(G), x^{p}-1)=1$. The rest follows from the fact that \[
 det(G)=\begin{vmatrix}
\sum_{i=0}^{p-1}f_{11i}x^{i} & \sum_{i=0}^{p-1}f_{12i}x^{i} \\
\sum_{i=0}^{p-1}f_{21i}x^{i} & \sum_{i=0}^{p-1}f_{22i}x^{i} 
\end{vmatrix} =\sum_{i=0}^{p -1}\left(\sum_{i\equiv r+s \mod p } f_{11r}f_{22s} - f_{12r} f_{21s}  \right)x^i.\] 
In particular, if $p$ be an odd prime such that $q$ is a primitive element in $\mathbb{F}_{p}$ and
\[\sum_{i=0} ^{p -1}\left(\sum_{i\equiv r+s \mod p } f_{11r}f_{22s} - f_{12r} f_{21s}  \right)\not\in \left\{0,\, p\sum_{{\substack{0\equiv r+s\mod p \\ 0\leq r \leq  s\leq p-1}}}\left(f_{11r}f_{22s} - f_{12r} f_{21s}\right)\right\};\]
then, by Lemma~\ref{corgusta}, $g(x)$ is PP. In fact, besides the condition 
\[\displaystyle{\sum_{i=0} ^{p -1}\left(\sum_{i\equiv r+s \mod p } f_{11r}f_{22s} - f_{12r} f_{21s}  \right)\neq 0},\]
the other condition on $g(x)$ to be a PP is  
\begin{eqnarray}
&&\left(1 -\frac{1}{p}\right)\sum_{0\equiv r+s\mod p}\left( f_{11r}f_{22s} - f_{12r} f_{21s}\right) -\frac{1}{p}\sum_{i=1} ^{p -1}\left(\sum_{i\equiv r+s \mod p } f_{11r}f_{22s} - f_{12r} f_{21s}  \right) \neq 0\nonumber \\
&\Leftrightarrow&\left(1-p \right)\sum_{0\equiv r+s\mod p}\left( f_{11r}f_{22s} - f_{12r} f_{21s}\right) + \sum_{i=1} ^{p -1}\left(\sum_{i\equiv r+s \mod p } f_{11r}f_{22s} - f_{12r} f_{21s}  \right)\neq 0\nonumber \\
&\Leftrightarrow& \sum_{i=0} ^{p -1}\left(\sum_{i\equiv r+s \mod p } f_{11r}f_{22s} - f_{12r} f_{21s}  \right)\neq p\sum_{{\substack{0\equiv r+s\mod p \\ 0\leq r\leq s\leq p-1}}}\left( f_{11r}f_{22s} - f_{12r} f_{21s}\right),\nonumber
\end{eqnarray}
and the result follows.
\end{enumerate}
\end{proof}

\begin{remark}
Since any linear polynomial over $\mathbb{F}_{q^{2p}}$ with coefficients in $\mathbb{F}_{q^{2}}$ can be represented as in Theorem~\ref{coefpol} (part $1$), being a PP depends exclusively on the coefficients $f_{jki}$ that appear on that representation, with $1\leq j, k \leq 2$ and $i=1,...,p-1$. Thus, it may be possible to construct new PP by imposing conditions on these coefficients that imply \[gcd\left(\sum_{i=0} ^{p -1}\left(\sum_{i\equiv r+s \mod p } f_{11r}f_{22s} - f_{12r} f_{21s}  \right)x^i, x^{p}-1 \right)=1,\]    
as it was done in the second part of Theorem \ref{coefpol} (part $2$). Example \ref{finalex} illustrates this fact.
\end{remark}

\begin{example}\label{finalex}
Let $p=3$ and $q=5$, then $\mathbb{F}_{q^{2p}}=\mathbb{F}_{5^{6}}$ and $5$ is a primitive element in $\mathbb{F}_{3}$.
In addition, $\displaystyle{\mathbb{F}_{q^2}=\mathbb{F}_{5^2}\cong \frac{\mathbb{F}_5 [x]}{\left\langle x^2 + 4x +2 \right\rangle}}$, where $f(x)=x^2 + 4x +2 $ is a primitive polynomial. Let $\alpha$ be one of the roots of $f$. Then $B=\left\{\alpha,\alpha^{q}=\alpha^5 \right\}$ is a $\mathbb{F}_{5}$-normal basis of $\mathbb{F}_{5^2}$ and $B'=\left\{u=\alpha^4 , u^{q}=\alpha^{20} \right\}$ is its respective dual basis. If 
\begin{eqnarray*}
g(x) &:=&\sum_{i=0} ^{2} g_i \left(x^{q^{2i}}\right)= g_0 \left(x\right)+ g_1 (x^{q^{2}})+ g_2 (x^{q^{4}})\\
                                                 &=& \left[ \left(3 \alpha^{q} + 1 \alpha \right)u + \left(3 \alpha^{q} + 1 \alpha \right)u^q  \right]x + \left[\left(3 \alpha^{q} + 1 \alpha \right)u + \left(3 \alpha^{q} + 1 \alpha\right)u^{q}  \right]x^q\\
                                                 &+& \left[ \left(1 \alpha^{q} + 1 \alpha \right)u + \left(3 \alpha^{q} + 1 \alpha \right)u^q  \right]x^{q^{2}} + \left[\left(3 \alpha^{q} + 1 \alpha \right)u + \left(1 \alpha^{q} + 1 \alpha\right)u^{q}  \right]x^{q^{3}}\\
                                                 &+& \left[ \left(4 \alpha^{q} + 4 \alpha \right)u + \left(1 \alpha^{q} + 2 \alpha \right)u^q  \right]x^{q^{4}} + \left[\left(1 \alpha^{q} + 2 \alpha \right)u + \left(4 \alpha^{q} + 4 \alpha\right)u^{q}  \right]x^{q^{5}}
\end{eqnarray*}
(which is represented as in Theorem \ref{coefpol}, part $1$), then
\[\sum_{i=0} ^{2}\left(\sum_{i\equiv r+s \mod 3 } f_{11r}f_{22s} - f_{12r} f_{21s}  \right)=4\not\in \left\{0,\, 3  \left(\sum_{{\substack{0\equiv r+s\mod 3  \\ 0\leq r \leq  s\leq 2}}}\left(f_{11r}f_{22s} - f_{12r} f_{21s}\right)\ \right)=3(1)=3\right\}.\]
Thus, by Theorem~\ref{coefpol} (part $2$), $g(x)$ is a PP.
\end{example}

\section{Conclusion}

In this work, some explicit descriptions of families of linear PPs in $\mathbb{F}_{q^{ms}}$ with coefficients in $\mathbb{F}_{q^{m}}$ were presented, and the sizes of some of these families were computed. These results extend and complement the results in \cite{bcv} and \cite{yuan2011,zhou2008}, respectively. In addition, several criteria to determine if linear polynomials of $\mathbb{F}_{q^{2p}}$ with coefficient in $\mathbb{F}_{q^{2}}$ are PPs were given. \\ 

The particular families here introduced as SPPs, BPPs, DPPs, SBPPs, and SDPPs, could be of use in applications to coding theory and cryptography. In particular, the simplicity of the shapes of these families of polynomials might be of use for studying the cycle structure of these permutations, which is a problem of interest for its various applications, among which is the construction of turbo codes (see \cite{gerike2020}). 

\end{document}